\documentclass[english, 12pt]{article}

\usepackage{amsfonts}
\usepackage{amsmath}
\usepackage{amssymb}
\usepackage{amsthm}
\usepackage{amscd}

\newtheorem{theo}{Theorem}[section]
\newtheorem{prop}{Proposition}[section]
\newtheorem{lem}{Lemma}[section]

\newtheorem{rk}{Remark}[section]

\newtheorem{coro}{Corollary}[section]
\numberwithin{equation}{section}

\newcommand{\N}{{\mathbb N}}
\newcommand{\Z}{{\mathbb Z}}

\newcommand{\R}{{\mathbb R}}

\newcommand{\calE}{{\cal E}}
\newcommand{\ds}{\displaystyle}
\newcommand{\la}{\lambda}
\newcommand{\be}{\beta}
\newcommand{\Om}{\Omega}
\newcommand{\cH}{\check{H}}
\newcommand{\ccalE}{\check{\cal{E}}}

\begin{document}
\title{Robin Laplacian in the Large coupling limit: Convergence and spectral asymptotic}
\author{\normalsize  Faten Belgacem\footnote{ Department of Mathematics, Faculty
of Sciences of Gab\`es. Uni. Gab\`es, Tunisia. }\ ,
Hichem BelHadjali\footnote{Department of Mathematics, I.P.E.I.N. Uni. Carthage, Tunisia. E-mail: hichem.belhadjali@ipein.rnu.tn}\ ,
Ali BenAmor\footnote{corresponding author} \footnote{Department of Mathematics, Faculty
of Sciences of Gab\`es. Uni. Gab\`es, Tunisia. E-mail: ali.benamor@ipeit.rnu.tn}
\ \& Amina Thabet\footnote{Department of Mathematics, Faculty
of Sciences of Gab\`es. Uni. Gab\`es, Tunisia.}
}
\date{}
\maketitle
\begin{abstract}
We study convergence modes as well as their respective rates for  the resolvent difference of Robin and Dirichlet Laplacian on bounded smooth domains in the large coupling limit. Asymptotic expansions for the resolvent,
the eigenprojections and the eigenvalues of the Robin  Laplacian are performed. Finally we apply our results to the case of the  unit disc.
\end{abstract}
\section{Introduction}
Let $\Omega\subset\mathbb{R}^{d}$ be an  open bounded domain, with smooth boundary $\Gamma$ and $\sigma$ the normalized surface measure on $\Gamma$.\\
We consider the bilinear symmetric form defined in $L^2(\Om):=L^2(\Om,dx)$ by
\begin{eqnarray}
D({\cal E}^{\beta})=H^1({\Omega}),\ {\cal E}^{\beta}(u,v):=
\int_{{\Omega}}\nabla u\cdot\nabla v\,dx+\beta\int_{\Gamma} uv\,d\sigma,\
\beta\geq 0.
\end{eqnarray}
Thanks to the continuity of the trace operator from  $H^{1}(\Omega)$ into $L^{2}(\Gamma,\sigma)$), the form $\calE^{\beta}$ is closed. Denote by $H_{\beta}$ the selfadjoint operator associated to $\cal E^{\beta}$ via Kato representation theorem. The operator $H_\beta$ is commonly named  the Laplacian with Robin boundary conditions and is characterized by
\begin{eqnarray}
D(H_\beta) = \big\{ u\in H^2(\Om),\ \frac{\partial u}{\partial\nu}+\beta u=0,  &\mbox{on}\ \Gamma\big\},\
H_\beta u = -\Delta u,\ {\rm on}\ \Om,
\end{eqnarray}
where  $\nu$ is  the outer normal unit vector  on $\Gamma$.\\
%
%
%
By Kato's monotone convergence theorem for sesquilinear  forms (see \cite[Theorem 3.13a, p.461 ]{Kato}), the bilinear symmetric forms $\calE^{\beta}$ increase, as $\beta$  increases to infinity, to the closed bilinear symmetric form $\calE^\infty$, defined by
\begin{eqnarray}
D(\calE^\infty)=\big\{u\in H^1(\Om),\ u=0,\ {\rm on}\ \Gamma\big\},\ \calE^\infty(u,v)=\int_\Om \nabla u\cdot\nabla v\,dx.
\end{eqnarray}
Thus $D(\calE^\infty)=H_0^1(\Om)$ and $\calE^\infty$ is nothing else but the quadratic form associated to the Dirichlet Laplacian in $L^2(\Om)$, which we denote by $-\Delta_D$. Thereby we obtain the strong convergence
\begin{eqnarray}
\lim_{\beta\to\infty}(H_{\beta} + 1)^{-1} = (-\Delta_D+1)^{-1},\ {\rm strongly}.
\end{eqnarray}
In a wide variety of applications it turns out that it
is more easy to analyze the limit than the approximating operators $(H_{\beta} + 1)^{-1}$. For
this reason one might use the following strategy for the investigation of the
operator $H_{\beta}$ for large $\beta$: One studies the limit of the operators $(H_{\beta} + 1)^{-1}$
and estimates the error one makes by replacing $(H_{\beta} + 1)^{-1}$ by the limit. This
leads to the question about how fast the operators $(H_{\beta} + 1)^{-1}$ converge. It
is also important to find out which kind of convergence takes place. For
instance convergence w.r.t. the operator norm admits much stronger conclusions
about the spectral properties than strong convergence, cf., e.g., the
discussion of this point in \cite{Reed}, chapter VIII.7.\\
In this spirit it is also interesting and practical to write down explicit asymptotic expansions for the operator $(H_\beta + 1)^{-1}$ and possibly for the eigenprojections and eigenvalues of the operator $H_\beta$ for large $\beta$.\\
On the light of these motivations, we shall establish, in these notes, operator norm convergence as well as convergence within Schatten--von Neumann ideals of $(H_{\beta} + 1)^{-1}$ towards $(-\Delta_D+1)^{-1}$ as $\beta\to\infty$. The speed of convergence for both convergence modes will be also determined.  Furthermore large coupling asymptotic for spectral objects will be performed.\\
An aspect of novelty at this stage, among others, is that we shall establish a second order asymptotic of the eigenvalues which coefficients are explicitly computed. In its own this expansion generalizes and push forward the one given in \cite{carbou} where the Neumann Laplacian with high conductivity inside $\Om$ is studied.\\
Let us emphasize that although we shall consider regular bounded domains, our method (which basically rests on the theory elaborated in \cite{brasche-demuth,benamor-brasche,BBB,bbb12,BBBMath}) still works for exterior domains with smooth boundary, regarding convergence of resolvents differences.\\
Physically the Laplacian with Robin boundary conditions describes the interaction of a particle inside $\Om$ with a potential of strength $\beta$ concentrated on the boundary $\Gamma$. Thus for large $\beta$ it describes the motion of a particle inside a set with high conductivity on the boundary (superconductivity on the boundary). We shall show among others that this phenomena is completely different from the case of having conductivity inside $\Om$, concerning convergence modes and convergence rates and hence spectral asymptotic.\\
The outline of this paper is as follows: In section 2 we give some preliminaries, whereas  in section 3 we prove uniform  convergence as well as convergence with respect to the Schatten--von Neumann norm of $(H_{\beta}+1)^{-1}-(-\Delta_D + 1)^{-1}$. The rate of convergence, with respect to both convergence types, is also discussed in this section. Section 4 and 5  are devoted to establish the asymptotic expansions for the resolvent, the projection and the eigenvalues of Laplacian with Robin boundary conditions for large coupling constant. In the last section  we work out the case where $\Om$ is the unit disc.
\section{Preliminary}
Along the paper we adopt the following notations:
\begin{itemize}
\item $K_{1}=(-\Delta_{N}+1)^{-1}$, where $-\Delta_N$ is the Neumann Laplacian on $\Om$.
\item $H_{\beta}$ is the selfadjoint operator in $L^2(\Omega)$ associated with ${\cal E}^{\beta}$.
\item $D_{\beta}=K_1-(H_{\beta}+1)^{-1}$
\item $D_{\infty}$ is the strong limit $lim_{\beta\rightarrow\infty}D_{\beta}=K_1 - (-\Delta_D + 1)^{-1}$.
\item $\calE[u]=\calE(u,u),\ \forall\,u \in H^1(\Om)$
\item $\calE_1 =\calE + (\cdot,\cdot)_{L^2(\Om)}$.
\end{itemize}
We designate by  $J$ the operator {\em  trace on the boundary} of functions from $H^1(\Om)$:
\begin{eqnarray}
J:(H^1(\Om),\calE_1)\to L^2(\Gamma):=L^2(\Gamma,\sigma),\ Ju=tr u.
\end{eqnarray}
As $\Gamma$ is smooth, it is  known that $Ran J=H^{1/2}(\Gamma)$, hence the operator $JJ^*$ is invertible. We set
\begin{eqnarray}
\cH:=(JJ^*)^{-1}.
\end{eqnarray}
Let us also recall that  the following Hardy type inequality holds true
\begin{eqnarray}
\int_\Gamma (Ju)^2\,d\sigma\leq c\calE_1[u],\ \forall\,u\in H^1(\Om).
\end{eqnarray}
and (see \cite{Adams,gilbarg})
\begin{eqnarray}
\ker(J)=H^1_0(\Omega)=\{u\in H^1(\Omega),u=0\ \mbox{on}\ \Gamma\}
\end{eqnarray}
We shall make extensive use of the following formulae, established in \cite[Lemma 2.3]{BBB}
\begin{eqnarray}
D_{\beta}=J^\ast (1+JJ^\ast)^{-1}JK_{1}=(JK_{1})^\ast (\frac{1}{\beta}+JJ^\ast)^{-1}(JK_{1}),\ \beta>0.
\label{resolvent}
\end{eqnarray}
and \cite[Lemma 2.4]{BBB}
\begin{eqnarray}
  D_{\infty}:=\lim_{\beta\to\infty}D_\beta=(\check{H}^{1/2}JK_1)^*\check{H}^{1/2}JK_1.
\label{Dinf}
\end{eqnarray}
Let $H_0^1(\Omega)^\perp$ be the $\calE_1$- orthogonal of $H_0^1(\Om)$ and $P$ be the $\calE_1$-orthogonal projection of $H^1(\Om)$ into $H_0^1(\Omega)^\perp$ . Then
$$
J|_{H_0^1(\Omega)^\perp}: H_0^1(\Omega)^\perp\rightarrow H^{1/2}(\Gamma)
$$
is an isomorphism. Its inverse operator, which we denote by $\mathcal{R}$ is given by
\begin{eqnarray}
\mathcal{R}:H^{1/2}(\Gamma)\to H^1(\Omega),\ \psi\mapsto Pv,\ Jv=\psi.
\end{eqnarray}
The operator $\mathcal{R}$ is well defined. Indeed, $\mathcal{R}u$ is the unique solution in $H^1(\Omega)$ of the boundary problem
\begin{eqnarray}
\left\{\begin{array}{ccccc}
-\Delta v+v & = & 0 & \mbox{in} & \Omega \\
v & = & \psi & \mbox{on} & \Gamma\\
\end{array}\right.
\end{eqnarray}
%
%
%
%
\begin{lem}
The operator $J$ is compact.
\label{compact}
\end{lem}
Although the result is known we shall give an alternative proof.
\begin{proof}
Owing to the regularity of  $\Om$ and precisely to the fact that
\begin{eqnarray}
\sigma(B_r(x)\cap\Gamma)\sim r^{d-1},\ \forall\,x\in\Gamma,\ 0<r<1,
\end{eqnarray}
the following known (see \cite[Theorem 5.36, p.164]{adams-fournier}) trace inequality holds true: For $d\geq 3,\ 2\leq p<\frac{2(d-1)}{d-2}$, there is a constant $c$ such that
\begin{eqnarray}
\big(\int_\Gamma |Ju|^p\,d\sigma\big)^{2/p}\leq c\calE_1[u],\ \forall\,u\in H^1(\Om),
\label{trace-emb}
\end{eqnarray}
whereas the latter inequality holds true for every $2\leq p<\infty$, for  $d=2$.\\
Now the compactness of $J$ follows from \cite[Theorem 4.1]{benamor07}.

\end{proof}
Of major importance for our method is the operator $JJ^*$, for which we list the relevant properties and give its precise description.\\
As $Ran(J)$ is dense in $L^2(\Gamma)$, the operator $JJ^*$ is an invertible nonnegative selfadjoint operator in $L^2(\Gamma)$. Set
\begin{eqnarray}
\check{H}:=(JJ^*)^{-1}.
\end{eqnarray}
Then $\check{H}$ is a nonnegative selfadjoint operator in $L^2(\Gamma)$ as well and has, by Lemma \ref{compact}, a compact resolvent.\\
In general it is hard to give a clear description of the domain of $\cH$. To overcome this difficulty  we shall  associate to $\cH$ a bilinear symmetric form, which domain is well known as well as its features.\\
Let us introduce the quadratic form $\check{{\cal E}}_1$ in $L^2(\Gamma)$, as follows:
\begin{eqnarray}
D(\check{{\cal E}}_1) = Ran(J),\ \check{{\cal E}}_1(Ju,Jv)= {\cal E}_1(Pu,Pv )~~\forall u,v\in H^1(\Om).
\end{eqnarray}
%
The operator  $\check{{\cal E}}_1$ is called the trace of the Dirichlet form ${\cal E}_1$ with respect to the measure $\sigma$ (see \cite[Chap. 6]{fukushima}). It is also called the Dirichlet--to--Neumann operator and was studied by many authors. For instance we refer the reader to \cite{fukushima-chen,Arendt,arendt-elster,daners,auchmuty} and references therein.\\
According to \cite[Theorem 1.1]{BBBMath}, the quadratic form $\ccalE_1$ is closed and is associated, in the sense of Kato's representation theorem,  to the selfadjoint operator $\check{H}^{-1}$. In this special context we shall give much accurate description of the operator $\check{H}$.
\begin{prop}
\begin{enumerate}
\item Let $\psi\in H^{1/2}(\Gamma)$ and  $u\in H^1(\Om)$ be the unique solution of the boundary value problem
\begin{eqnarray}
\left\{\begin{array}{ccccc}
-\Delta u+u & = & 0 & \mbox{in} & \Omega \\
u & = & \psi & \mbox{on} & \Gamma\\
\end{array}\right.
\end{eqnarray}
Then $\ccalE_1[\psi]=\calE_1[u]$. Furthermore for every $\psi\in D(\check{H})$, $\check{H}\psi=\frac{\partial u}{\partial\nu}$.
\item (Dirichlet principle). For every $\psi\in H^{1/2}(\Gamma)$, we have
\begin{eqnarray}
\ccalE_1[\psi]=\inf\big\{\calE_1[v]\colon\ v\in H^1(\Om),\ Jv=\psi\big\}.
\label{DP}
\end{eqnarray}
It follows that $\ccalE_1$ is a Dirichlet form.
\item For every $u\in L^2(\Gamma)$, set $U_1^\sigma u$ the $1$-potential of the signed  measure $u\sigma$. Then $\check{H}^{-1}u=JU_1^\sigma u$.
\item Let $G_\Om$ be the Neumann function of $-\Delta +1$ on $\Om$, i.e. the function $G:\Omega\times\Omega\to [0,\infty]$ satisfying
\begin{eqnarray}
\left\{\begin{array}{ccccc}
-\Delta_y G(\cdot,y)+G(\cdot,y) & = & \delta_{\cdot}(y) & \mbox{on} & \Omega \\
\frac{\partial G(\cdot,y)}{\partial\nu} & = & 0 & \mbox{on} & \Gamma\\
\end{array}\right.
\end{eqnarray}
Define the operator
\begin{eqnarray}
K_\Om^\sigma:L^2(\Gamma)\to L^2(\Gamma),\ \psi\mapsto \int_\Gamma G_\Om(\cdot,y)\psi(y)\,d\sigma(y).
\end{eqnarray}
\end{enumerate}
Then $\check{H}^{-1}=K_\Om^\sigma$.
\label{description}
\end{prop}
\begin{proof}
Assertion 1. follows from the very construction of $\ccalE_1$ and the use of Green's formula.\\
2. Clearly the left-hand-side of (\ref{DP}) is bigger than its right-hand-side. The reversed inequality follows from the existence of a minimizer together with the continuity of $J$.\\
3. Let us first observe that for every fixed $u\in H^1(\Om)$ the signed measure $u\sigma$ has finite energy integral, i.e.
\begin{eqnarray}
|\int_\Gamma vu\,d\sigma|\leq c(\calE_1[v])^{1/2},\ \forall\,v\in H^1(\Om).
\end{eqnarray}
Thus the $1$-potential of $u\sigma$ is well defined and is characterized as being the unique element of $H^1(\Om)$ such that
\begin{eqnarray}
\calE_1(U_1^\sigma u,v)=\int_\Gamma JuJv\,d\sigma,\ \forall\,v\in H^1(\Om).
\end{eqnarray}
Hence
\begin{eqnarray}
\ccalE_1(JU_1^\sigma,Jv)&=&\calE_1(U_1^\sigma u,Jv)=\int_\Gamma JuJv\,d\sigma\nonumber\\
&=&\ccalE_1(\cH^{-1}Ju,Jv), \forall\,v\in H^1(\Om).
\end{eqnarray}
4: Follows from the fact that
\begin{eqnarray}
(-\Delta_N +1)^{-1}u= \int_\Om G_\Om(\cdot,y)u(y)\,dy,\ \forall\,u\in H^1(\Om),
\end{eqnarray}
and the identity $\cH^{-1}=JJ^*$.
\end{proof}
Henceforth we denote by $e^{-tT},\ t>0$, respectively  $\check{T_t},\ t>0$, the semigroup associated to $-\Delta_N +1$, respectively to $\cH$.
\begin{rk}
{\rm
From potential theoretical results relating properties of $(T_t)_{t>0}$ to those of $(\check{T_t})_{t>0}$, we conclude on the light of the latter Proposition that  $(\check{T_t})_{t>0}$ is Markovian and transient, however it is not  conservative, i.e.,
\begin{eqnarray}
0\leq \check{T_t}1\neq 1,\ \forall\,t>0
\end{eqnarray}

}
\end{rk}

\section{Uniform and trace class convergence}
%
In this section we shall concentrate on various types of convergence of $D_\beta$ to $D_\infty$ as well as their rates. These types are precisely convergence with respect to the operator norm and the norms of Schatten--von Neumann ideals.\\
Let us first quote that $\lim_{\beta\to\infty}\|D_\beta -D_\infty\|=0$. Indeed, we already mentioned that $D_\beta$ increases strongly to $D_\infty$ which is compact. Thus using \cite[Theorem 2.6]{BBB} we get uniform convergence.
\begin{theo} The operator $\check{H}JK_1$ is bounded. Consequently $(H_\beta+1)^{-1}$ converges in the operator norm to $(-\Delta_D+1)^{-1}$ with maximal rate proportional to $\beta^{-1}$. Moreover,
\begin{eqnarray}
\lim_{\beta\to\infty}\beta\|D_\beta - D_\infty\|=\|\cH JK_1\|^2.
\label{rateU}
\end{eqnarray}
\label{UniConv}
\end{theo}
\begin{proof}
Let  $u\in H^2(\Omega)$. We claim that  ${P}u\in H^2(\Omega)$. Indeed, $J{P}u=Ju$. Therefore
${P}u$ is the unique solution in $H^1(\Omega)$ of the boundary problem:
\begin{eqnarray}
\left\{\begin{array}{ccccc}
-\Delta v+v & = & 0 & \mbox{in} & \Omega \\
v & = & u & \mbox{on} & \Gamma\\
\end{array}\right.
\end{eqnarray}
From elliptic regularity (see \cite[Theorem 8.13]{gilbarg}), we  get that $Pu\in H^2(\Omega)$ and the claim is proved.\\
Let $u\in L^2(\Omega)$ and $v\in H^1(\Omega)$. Then
\begin{eqnarray}
\check{\mathcal{E}}_1(JK_1u,Jv)) &=& \mathcal{E}_1(PK_1u,Pv)\nonumber\\
& = & \int_\Omega (\nabla
PK_1u){\nabla Pv}
+\int_{\Omega} (PK_1u) {Pv}.
\end{eqnarray}
As $K_1 u\in H^2(\Omega)$ then  $PK_1u\in H^2(\Omega)$ as well. Thus by  Green's formula one obtain
\begin{eqnarray}
\begin{array}{ccccc}\check{\mathcal{E}}_1(JK_1u,Jv)) & =  &\ds -\int_\Omega \Delta PK_1u {Pv} &+&%
\ds\int_{\Gamma}\frac{\partial PK_1u}{\partial\nu}{Pv} +\int_\Omega PK_1u{Pv }\\
& = & \ds\int_{\Gamma}\frac{\partial PK_1u}{\partial\nu}{v} &=&\ds (\frac{\partial PK_1u}{\partial\nu},Jv)_{L^2(\Gamma)}.
\end{array}
\end{eqnarray}
It follows that $JK_1u\in D(\check{H})$ and
$\check{H}JK_1u=\frac{\partial PK_1u}{\partial\nu}$. Thus $\check{H}JK_1$ is a closed everywhere defined operator on $L^2(\Omega)$ and hence is bounded.\\
Finally utilizing  \cite[Theorem 2.7]{BBB}, we conclude that $(H_{\beta}+1)^{-1}$
converges uniformly to $(-\Delta_D+1)^{-1}$ with maximal rate proportional to $\frac{1}{\beta}$ and that formula (\ref{rateU}) holds true.
\end{proof}
\begin{rk}{\rm
Here we observe a qualitative difference between inner superconductivity and boundary superconductivity: Whereas in our setting uniform convergence is as fast as possible, it occurs for $-\Delta +\beta 1_{\Om_1}$, where $\Omega$ is open and $\overline\Omega\subset\Om$, with a rate which is $O(\beta^{-1/2})$, according to \cite{carbou,agbanusi}.
}
\end{rk}
For further investigations concerning  convergence of resolvent differences as well as spectral asymptotic one needs strengthened regularizing properties of the operator $JK_1$. To that end we establish:
\begin{lem}\label{cor2.1}
The operator $\check{H}^{3/2}JK_1$ is bounded.
\end{lem}
\begin{proof} Let $u\in L^2(\Omega)$. We have already proved that
$\ds\check{H}JK_1u=\frac{\partial PK_1u}{\partial\nu}$, which by elliptic regularity lies in the space $
H^{1/2}(\Gamma)=D(\check{H}^{1/2})$.\\
Thus $\check{H}^{3/2}JK_1$ is a closed everywhere defined operator on $L^2(\Omega)$ and  then it is bounded.
\end{proof}
Before dealing with convergence within Schatten--von Neumann operator ideals, let us introduce few notations.\\
Let $1\leq p<\infty$ and ${\cal{H}}_i$ be Hilbert spaces, $i=1,2$. Let $K:{\cal{H}}_1\to {\cal{H}}_2$ be a compact operator. Then ${\cal{H}}_2$ has an orthonormal basis $(e_i)_{i\in I}$ such that, with $|K|:=\sqrt{KK^*}$, we have
$$
|K|e_i=\lambda_i e_i,\ \forall\,i\in I,
$$
for some suitably chosen family $(\lambda_i)_{i\in I}\subset [0,\infty)$, which is unique up to permutations. We set
$$
 \|K\|_{S_p}:=\big(\sum_{i\in I}\lambda_i^p\big)^{1/p}.
$$
The ideal $S_p({\cal{H}}_1,{\cal{H}}_2)$, ($S_p$ for short) denotes the set of compact operators from ${\cal{H}}_1$ to ${\cal{H}}_2$ such that $\|K\|_p<\infty$. It is called the Schatten--von Neumann class of order $p$.

On the light of Lemma \ref{cor2.1}, we are able to prove convergence with respect to the $S_p$ norm.
\begin{theo}
For every  $p>\frac{d-1}{2}$ we have
\begin{eqnarray}
\lim_{\beta\to\infty}\|D_\beta-D_\infty\|_{S_p}=0.
\end{eqnarray}
In particular trace class convergence  holds true for $d=2$.
\end{theo}
\begin{proof} First we recall that owing to \cite[Corollary 2.20]{BBB} that $S_p$-convergence holds true whenever $D_\infty\in S_p$.\\
Having in mind that $D_\infty=(\check{H}^{1/2}JK_1)^*\check{H}^{1/2}JK_1$, we get that it  lies in $S_p$ if and only if $\check{H}^{1/2}JK_1$ lies in $S_{2p}$. On the other hand from the boundedness of $\check{H}^{3/2}JK_1$, we obtain
\begin{eqnarray}
\|\check{H}^{1/2}JK_1\|_{S_{2p}}\leq \|\check{H}^{-1}\|_{S_{2p}}\|\check{H}^{3/2}JK_1\|.
\end{eqnarray}
Thus we are led to prove that $\check{H}^{-1}\in S_{q}$ for $q>d-1$.\\
To that end we shall use the trace  inequality (\ref{trace-emb}) to obtain:\\
a) For $d\geq 3$: from the construction of $\ccalE$, the following Sobolev type  inequality holds true
\begin{eqnarray}
\big(\int_\Gamma |\psi|^{\frac{2(d-1)}{d-2}}\,d\sigma\big)^{\frac{d-2}{d-1}}\leq C\check{\calE_1}[\psi],\ \forall\,\psi\in H^{1/2}(\Gamma).
\label{Sob}
\end{eqnarray}
Now it is standard that Sobolev inequality leads to a lower bound for the eigenvalues $\check{\lambda}_k$ (repeated as many times as their multiplicity in an increasing way) of $\check H$:
\begin{eqnarray}
\check{\lambda}_k\geq Ck^{\frac{1}{d-1}}.
\label{Low-Eig1}
\end{eqnarray}
Thus $\check{H}^{-1}\in S_q$ for every $q>d-1$, which was to be proved for $d\geq 3$.\\
b) For $d=2$, the proof is similar so we omit it.

\end{proof}
By the end of this section we shall discuss the rate of convergence in $S_1$ in two dimensions. It was proved in \cite[Theorem 2.3]{bbb12} that the maximal rate of convergence in $S_1$ is proportional to $1/\beta$ and that trace-class convergence holds true if and only if the operator $\check{H}JK_1$ is a Hilbert--Schmidt operator. However, according to \cite[Proposition 2.4]{bbb12} if for some $r\in(0,1)$ the operator $\check{H}^{\frac{1+r}{2}}JK_1$ is a Hilbert--Schmidt operator then one has trace-class convergence with rate $O(1/{\beta^r})$.
\begin{prop}
In the case $d=2$ it holds
\begin{eqnarray}
\lim_{\beta\to\infty}\beta^r\|D_\beta-D_\infty\|_{S_1}<\infty,\ \forall\,r\in(0,1).
\end{eqnarray}
\label{2D}
\end{prop}
\begin{proof}
For $d=2$ we have the lower bound
$$
\check{\lambda_k}\geq Ck^{\frac{p-2}{p}},\ \forall\,p>2.
$$
Thus if for a given  $r\in(0,1)$, we choose $p>2\frac{2-r}{1-r}>2$, we get $(2-r)\frac{p-2}{p}>2$. Thus $\check{H}^{\frac{r-2}{2}}$ is a Hilbert--Schmidt operator and so is $\check{H}^{\frac{1+r}{2}}JK_1$.
\end{proof}
\begin{rk}
{\rm
We shall show in the example below that the limit exponent $r=1$ is excluded!
}
\end{rk}

\section{Asymptotic expansions for the resolvents and the eigenprojections}
Asymptotic expansions are theoretically and numerically interesting in the sense that they offers  'good' approximations for the studied objects. How 'good' is the approximation depends on its order and on the computation of its coefficients. In operator theory there are two types of asymptotic: uniform, i.e. the rest is small with respect the operator norm and strong asymptotic, i.e., the rest is small for  every fixed element from the domain of the operator.\\
Although we shall give lower order asymptotic (of second order) of the spectral objects related to Robin Laplacian,  we shall write explicitly the  coefficients of the asymptotic and this is new to our best knowledge for such problems. In particular we shall show that the coefficients involved in the asymptotic depend only on the Neumann Laplacian and its trace, the Dirichlet Laplacian and the trace operator.\\ Especially, the coefficients of the expansion of the eigenvalues of the Robin Laplacian depend only on the eigenvalues of the Dirichlet Laplacian.
\\
We start by giving an asymptotic expansion for $(H_\beta + 1)^{-1}$.
\begin{theo} The following first order uniform expansion holds true:
\begin{eqnarray}\label{first asymp formmula}
(H_\beta+1)^{-1}=(-\Delta_D+1)^{-1}+\frac{1}{\beta}K+\frac{1}{\beta^2}K'.
\end{eqnarray}
where, $\ds
K=(\check{H}JK_1)^*\check{H}JK_1=\mathcal{R}\frac{\partial PK_1}{\partial\nu}$
and $\|K'\|\leq \|\check{H}^{3/2}JK_1\|^2$.
\end{theo}
\begin{proof}
From the construction of $\ccalE_1$ we derive
\begin{eqnarray}
\check{\mathcal{E}}_1(JK_1u,Jv))=
\mathcal{E}_1(K_1u,Pv)=(u,
Pv)_{L^2(\Omega)},\forall u\in L^2(\Omega), \forall v\in
H^1(\Omega).
\end{eqnarray}
It follows that
\begin{eqnarray}
(\check{H}JK_1u,Jv)_{L^2(\Gamma)}=(u,
Pv)_{L^2(\Omega)}
\ and\ (\check{H}JK_1)^*Jv=Pv.
\end{eqnarray}
Then,
\begin{eqnarray}
(\check{H}JK_1)^*\check{H}JK_1u=P\mathcal{R}\frac{\partial PK_1u}{\partial\nu}=
\mathcal{R}\frac{\partial PK_1u}{\partial\nu}.
\end{eqnarray}
On the other hand relying on the resolvent formula (\ref{resolvent})we obtain
\begin{eqnarray}
%
\ds D_\infty-D_\beta
  &=&    (\check{H}^{1/2}JK_1)^*\check{H}^{1/2}JK_1-(JK_1)^*(\frac{1}{\beta}+\check{H}^{-1})^{-1}JK_1\\
  &=& \ds(\check{H}^{1/2}JK_1)^*\check{H}^{1/2}JK_1-
  (\cH^{1/2}JK_1)^*(1+\frac{1}{\beta}\check{H})^{-1}\check{H}^{1/2}JK_1\\
  &=& \frac{1}{\be}(\check{H}JK_1)^*(1+\frac{1}{\beta}\check{H})^{-1}\check{H}JK_1\\
  &=& \frac{1}{\beta}(\check{H}JK_1)^*\check{H}JK_1-\frac{1}{\beta^2}(\check{H}^{3/2}JK_1)^*(1+\frac{1}{\beta}\check{H})^{-1}\check{H}^{3/2}JK_1\\
%
\end{eqnarray}
To conclude, it suffices to note that, $\ds 0\leq(1+\frac{1}{\beta}\check{H})^{-1}\leq 1$, and the proof is finished.
\end{proof}
Henceforth, $ o_s(\frac{1}{\be^2})$ (resp. $o_u(\frac{1}{\be^2})$ ) denotes an operator-valued function such
that $\be^2 o_s(\frac{1}{\be^2})f\to 0,\ \forall f$ (resp. $\be^2\|o_u(\frac{1}{\be^2})\|\to 0$) as $\be\to\infty$.\\
The latter theorem yields automatically  the second order strong asymptotic expansion for large $\beta$.
\begin{coro} For large $\beta$ the following strong asymptotic formula holds true:
\begin{eqnarray}
(H_\beta+1)^{-1}&=&(-\Delta_D+1)^{-1}+\frac{1}{\beta}(\check{H}JK_1)^*\check{H}JK_1-\frac{1}{\beta^2}
(\check{H}^{3/2}JK_1)^*\check{H}^{3/2}JK_1\nonumber\\
&+& o_s(\frac{1}{\beta^2}).
\label{order2}
\end{eqnarray}
\end{coro}
%
%
%
We turn our attention now to give the expansions of the eigenprojections. To that end we need an expansion for $(H_\beta -z)$ for $z$ in the resolvent set $\rho(H_\beta)$.\\
Since $\{(H_\beta+1)^{-1}\}$ converges in norm to
$(-\Delta_D+1)^{-1}$ when $\beta\to\infty$, it follows that if
$z\in\rho(-\Delta_D)$, then $z\in\rho(H_\beta)$ for $\beta$
sufficiently large and $\{(H_\beta-z)^{-1}\}$ converge in
norm to $(-\Delta_D-z)^{-1}$ uniformly in any compact subset of
$\rho(-\Delta_D)$ as $\beta$ goes to infinity. In particular the
family of the resolvents $\{(H_\beta-z)^{-1}\}$ is bounded
uniformly in $\be$ and $z$ in any compact subset of
$\rho(-\Delta_D)$ (for large $\beta$). Moreover, one has :
\begin{prop}
For large $\beta$, the resolvent $(-\Delta_\beta-z)^{-1}$ admits the second order strong asymptotic expansion uniformly in any compact subset of $\rho(-\Delta_D)$:
\begin{eqnarray}\label{second asymp formula}
(H_\beta-z)^{-1}&=&(-\Delta_D-z)^{-1}+\frac{1}{\beta}LKL-\frac{1}{\be^2}(LRL-(1+z)LKLKL)\nonumber\\
&&+o_s(\frac{1}{\beta^2}),
\end{eqnarray}
where $K$ is the operator given by Theorem \ref{first asymp formmula} and
\begin{eqnarray}
L=L(z):=\left(1+(1+z)(-\Delta_D-z)^{-1}\right),\ R:=(\check{H}^{3/2}JK_1)^*\check{H}^{3/2}JK_1.
\end{eqnarray}
\end{prop}
\begin{proof}

Let $z\in\rho(-\Delta_D)$, then for large $\beta$ one has
\begin{eqnarray}\label{res. diff}
D_\infty(z)-D_\beta(z)&=&(1+(1+z)(H_\beta-z)^{-1})(D_\infty-D_\beta)\nonumber\\
&&\cdot(1+(1+z)(-\Delta_D-z)^{-1})
\end{eqnarray}
By formula (\ref{order2}), it follows that:
\begin{eqnarray}
u-\lim_{\be\to\infty}\beta\,(D_\infty(z)-D_\beta(z))=LKL,
\end{eqnarray}
uniformly in any compact subset of $\rho(-\Delta_D)$.\\
Thus, one writes,
\begin{eqnarray}\label{order 1}
(H_\beta-z)^{-1}=(-\Delta_D-z)^{-1}+\frac{1}{\beta}LKL+o_u(\frac{1}{\beta})
\end{eqnarray}
Then, if we substitute $(D_\infty-D_\beta)$  and
$(-\Delta_\beta-z)^{-1}$ by the corresponding terms given by the
formulae (\ref{order2}) and (\ref{order 1}) respectively, in
the equation (\ref{res. diff}) we obtain the desired result.
\end{proof}
The operators $H_\beta$ converge to $-\Delta_D$ in the
norm resolvent sense, furthermore these operators are selfadjoint,
nonnegative with compact resolvents, then the eigenvalues of
$(H_\beta)$ converge to ones of
$-\Delta_D$.\\
Let $\lambda_\infty$ be an eigenvalue of $-\Delta_D$, since the
operator $-\Delta_D$ has compact resolvents, then there exists
$\epsilon>0$ such that: $\mbox{spec}(-\Delta_D)\cap
B(\lambda_\infty,\epsilon)=\{\lambda_\infty\}$, where $B(\lambda_\infty,\epsilon)=\{z\in\mathbb{C}, |z-\lambda_\infty|\leq\epsilon\}$ .\\
In the following we denote,
\begin{itemize}
\item $E_\infty=\ker(-\Delta_D-\lambda_\infty)$, the eigenspace of $\lambda_\infty$, and $P_\infty$ the
spectral projection onto $E_\infty$. It is known that
\begin{eqnarray}
P_\infty=-\frac{1}{2i\pi}\int_{C(\lambda_\infty,\epsilon)}(-\Delta_D-z)^{-1}\,
\end{eqnarray}

where, $C(\lambda_\infty,\epsilon)$ is the circle of center
$\lambda_\infty$ and of radius $\epsilon$.
\item $E_\beta$ is the direct sum of the eigenspaces associated to the
eigenvalues of $(H_\beta)$ contained in
$B(\lambda_\infty,\epsilon)$, and $P_\beta$ is the spectral
projection onto $E_\beta$ given by:
$$P_\beta=-\frac{1}{2i\pi}\int_{C(\lambda_\infty,\epsilon)}(H_\beta-z)^{-1}\,dz$$
\end{itemize}
\begin{prop}\label{proj asymp formula}
The spectral projection $P_\beta$ admits a uniform asymptotic expansion
of the form,
\begin{eqnarray}
P_\beta=P_\infty+\frac{1}{\beta}\,Q-\frac{1}{\be^2}Q_1+o_s(\frac{1}{\beta^2})
\end{eqnarray}
Moreover, $P_\infty Q P_\infty=0$.
\end{prop}
\begin{proof}
Setting
\begin{eqnarray}
Q=-\frac{1}{2i\pi}\int_{C(\lambda_\infty,\epsilon)}LKL\,dz,\
Q_1=-\frac{1}{2i\pi}\int_{C(\lambda_\infty,\epsilon)}(LRL-(1+z)LKLKL)\,dz,
\end{eqnarray}
then the first identity is immediate by integrating the formula
(\ref{second asymp formula}) along the circle $C$.\\
For the second identity, since $\lim_{\beta\to\infty}\|P_\beta -P_\infty\|=0$, we obtain
\begin{eqnarray}
P_\infty QP_\infty=\lim_{\beta\to\infty}\beta\,P_\infty(P_\beta-P_\infty)P_\beta
=0.
\end{eqnarray}
\end{proof}
\section{Asymptotic expansion for the eigenvalues}

Next we shall improve the asymptotic expansion of eigenvalues developed in \cite[Theorem 1.2]{carbou} and extend it  to our context which deals with singular perturbations. The novelty at this stage is that we give a second order asymptotic expansion which coefficients are given by the eigenvalues of a matrix depending only on the Dirichlet Laplacian.\\
To that end we need some intermediate results.
\begin{prop} The following formulae hold true:
\begin{enumerate}
\item
\begin{eqnarray}\label{Dynkin's formula}
\ds PK_1=(-\Delta_N+1)^{-1}-(-\Delta_D+1)^{-1}=D_\infty.
\end{eqnarray}
In particular,
\begin{eqnarray}
\ds\frac{\partial PK_1}{\partial\nu}=-\frac{\partial
(-\Delta_D+1)^{-1}}{\partial\nu}.
\end{eqnarray}
\item $\ds P_\infty
LKLP_\infty=\frac{1}{(\lambda_\infty-z)^2}MP_\infty$, where $M$ is the matrix with entries
\begin{eqnarray}
\left(\ds\int_{\Gamma}\frac{\partial
f_i}{\partial\nu}{\frac{\partial
f_j}{\partial\nu}\,d\sigma}\right)_{1\leq i,j\leq m}
\end{eqnarray}
in an orthonormal basis $(f_1,\cdots,f_m)$ of $E_\infty$.
\end{enumerate}
\label{dynkin}
\end{prop}
\begin{proof} For every $\forall u\in L^2(\Omega), PK_1u$ is the unique solution
of the boundary value problem
\begin{eqnarray}
\left\{\begin{array}{ccc}
                         - \Delta v+v=0 & in & \Omega \\
                           v = K_1u & in & \Gamma \\
                         \end{array}\right.
\end{eqnarray}
Let $v_0$ be the unique solution of the equation $-\Delta v+v=-u$, in $H^2(\Omega)\cap H_0^1(\Omega)$, that is
$v_0=(-\Delta_D+1)^{-1}(-u)$. Then, $PK_1u$ is given by:
\begin{eqnarray}
PK_1u=v_0+K_1u=(-\Delta_N+1)^{-1}u-(-\Delta_D+1)^{-1}u,
\end{eqnarray}
yielding the first assertion.\\
Let $f_i,f_j$ be  eigenfunctions associated to the eigenvalue $\lambda_\infty$ of $-\Delta_D$. Since $
L(z)f_i=(\frac{\lambda_\infty+1}{\lambda_\infty-z})\,f_i$, straightforward computations yields
\begin{eqnarray}
\begin{array}{lll}
(P_\infty LKLP_\infty
f_i,f_j)&=&(KL(z)f_i,L({z})f_j)_{L^2(\Omega)}\\
&=&\ds(\frac{\lambda_\infty+1}{\lambda_\infty-z})^2(\check{H}JK_1f_i,\check{H}JK_1f_j)_{L^2(\Gamma)}\\
&=&\ds(\frac{\lambda_\infty+1}{\lambda_\infty-z})^2(\frac{\partial
(-\Delta_D+1)^{-1}f_i}{\partial\nu},\frac{\partial
(-\Delta_D+1)^{-1}f_j}{\partial\nu})_{L^2(\Gamma)}\\
&=&\ds\frac{1}{(\lambda_\infty-z)^2}(\frac{\partial
f_i}{\partial\nu},\frac{\partial f_j}{\partial\nu})_{L^2(\Gamma)},
\end{array}
\end{eqnarray}
and the proof is done.
\end{proof}
\begin{prop}
Let $(f_k)$ be an orthonormal basis of Dirichlet eigenfunctions,
$$-\Delta_D f_k=\la_k f_k,$$
and $Q$ be the operator given by Proposition \ref{proj asymp formula}.\\
For $f_i$ and $f_j$ in $E_\infty$ we set,
\begin{eqnarray}
a_{i,j,k}:=(\frac{\partial f_i}{\partial\nu},\frac{\partial
f_k}{\partial\nu})_{L^2(\Gamma)}({\frac{\partial
f_j}{\partial\nu},\frac{\partial f_k}{\partial\nu}})_{L^2(\Gamma)}
\end{eqnarray}
Then we obtain,\\
\begin{enumerate}
\item $(P_\infty LRLP_\infty f_i,f_j)=\ds\frac{1}{(\lambda_\infty-z)^2}%
(\frac{\partial}{\partial\nu}{\cal{R}}(\frac{\partial
f_i}{\partial\nu}),\frac{\partial
f_j}{\partial\nu})_{L^2(\Gamma)}$
\item $(P_\infty LKLKLP_\infty f_i,f_j)=\ds\frac{1}{(z-\la_\infty)^2}\sum_k
\frac{1}{(\la_k-z)(1+\la_k)}\,a_{i,j,k}$
\item $(P_\infty LKLQP_\infty
f_i,f_j)=\ds\sum_{f_k\in
E_\infty^\perp}\frac{1}{(\la_\infty-z)(\la_k-z)(\la_k-\la_\infty)}\,a_{i,j,k}$
\end{enumerate}
\end{prop}
\begin{proof}
\begin{enumerate}
\item
\begin{eqnarray}
\begin{array}{lll}
(P_\infty LRLP_\infty
f_i,f_j)&=&\ds(RL(z)f_i,L({z})f_j)=(\frac{1+\lambda_\infty}{\lambda_\infty-z})^2(Rf_i,f_j)\\
&=&\ds(\frac{1+\lambda_\infty}{\lambda_\infty-z})^2(\check{H}^{3/2}JK_1f_i,\check{H}^{3/2}JK_1f_j)_{L^2(\Gamma)}\\
&=&\ds\frac{1}{(\lambda_\infty-z)^2}(\check{H}^{1/2}\,\frac{\partial f_i}{\partial\nu},\check{H}^{1/2}\,\frac{\partial f_j}{\partial\nu})_{L^2(\Gamma)}\\
&=&\ds\frac{1}{(\lambda_\infty-z)^2}(\check{H}\,\frac{\partial f_i}{\partial\nu},\frac{\partial f_j}{\partial\nu})_{L^2(\Gamma)}\\
&=&\ds\frac{1}{(\lambda_\infty-z)^2}(\frac{\partial}{\partial\nu}{\cal{R}}(\frac{\partial f_i}{\partial\nu}),\frac{\partial f_j}{\partial\nu})_{L^2(\Gamma)}.
\end{array}
\end{eqnarray}
In the last step we used the fact that for  $\varphi\in D(\check{H})$ we have $\check{H}\varphi=\frac{\partial
u}{\partial\nu}$ where $u={\cal{R}}\varphi$. Indeed, by the definition of ${\cal{R}}$,  ${\cal{R}}\varphi$  solves
\begin{eqnarray*}
\left\{\begin{array}{ccc}
                         - \Delta u+u=0 & in & \Omega \\
                           u = \varphi & in & \Gamma \\
                         \end{array}\right.
\end{eqnarray*}
\item Making use of Proposition \ref{dynkin}, an elementary computation yields
\begin{eqnarray}
(Kf_i,f_k)=(\check{H}JK_1f_i,\check{H}JK_1f_k)=(1+\la_i)^{-1}(1+\la_k)^{-1}(\frac{\partial
f_i}{\partial\nu},\frac{\partial f_k}{\partial\nu}).
\end{eqnarray}
Thus
\begin{eqnarray}
 Kf_i=\sum_k (1+\la_i)^{-1}(1+\la_k)^{-1}(\frac{\partial
f_i}{\partial\nu},\frac{\partial
f_k}{\partial\nu})f_k
\end{eqnarray}

and
\begin{eqnarray}
 L(z)Kf_i=\sum_k (1+\la_i)^{-1}(\la_k-z)^{-1}(\frac{\partial
f_i}{\partial\nu},\frac{\partial f_k}{\partial\nu})f_k.
\end{eqnarray}
Set $B=P_\infty H(z)KH(z)KH(z)P_\infty$. Then
\begin{eqnarray}
\begin{array}{lll}
(Bf_i,f_j)&=&\ds (\frac{1+\la_\infty}{z-\la_\infty})^2(KLKf_i,f_j)=(\frac{1+\la_\infty}{z-\la_\infty})^2(LKf_i,Kf_j)\\
&=&\ds \frac{1}{(z-\la_\infty)^2}\sum
\frac{1}{(\la_k-z)(1+\la_k)}(\frac{\partial
f_i}{\partial\nu},\frac{\partial
f_k}{\partial\nu})({\frac{\partial
f_j}{\partial\nu},\frac{\partial f_k}{\partial\nu}}).
\end{array}
\end{eqnarray}
\item Finally setting, $A(z,s)=P_\infty L(z)KL(z)L(s)KL(s)P_\infty$, we obtain
\begin{eqnarray}
\begin{array}{lll} (Af_i,f_j)&=&\ds
\frac{(1+\la_\infty)^2}{(\la_\infty-z)(\la_\infty-s)}(L(s)Kf_i,L({z})Kf_j)\\
&=&\ds \frac{1}{(\la_\infty-z)(\la_\infty-s)}\sum_k
\frac{1}{(\la_k-z)(\la_k-s)}(\frac{\partial
f_i}{\partial\nu},\frac{\partial
f_k}{\partial\nu})({\frac{\partial
f_j}{\partial\nu},\frac{\partial f_k}{\partial\nu}}).
\end{array}
\end{eqnarray}
Regarding the definition of $Q$, we achieve
\begin{eqnarray}
\begin{array}{lll} (P_\infty L(z)KL(z)QP_\infty
f_i,f_j)&=&-\ds\frac{1}{2i\pi}\int_{C(\lambda_\infty,\epsilon)}(A(z,s)f_i,f_j)\,d\sigma(s)\\
&=& \ds\sum_{f_k\in
E_\infty^\perp}\frac{(\la_k-\la_\infty)^{-1}}{(\la_\infty-z)(\la_k-z)}(\frac{\partial
f_i}{\partial\nu},\frac{\partial
f_k}{\partial\nu})({\frac{\partial
f_j}{\partial\nu},\frac{\partial f_k}{\partial\nu}})
\end{array}
\end{eqnarray}
 \end{enumerate}
\end{proof}
Now we are in position to establish the asymptotic of the eigenvalue of the Robin Laplacian.
\begin{theo}\label{asym expan of eigenvalues}
Let $\lambda_\infty$ be an eigenvalue of $-\Delta_D$ with multiplicity $m$ and eigenspace $E_\infty$. Then for sufficiently large $\beta$, the operator $H_\beta$ has exactly $m$ eigenvalues counted according to their multiplicities in $B(\lambda_\infty,\epsilon)$. These eigenvalues admit the asymptotic expansion
\begin{eqnarray}
\lambda_{i,j,\beta}=\lambda_\infty-\frac{1}{\beta}\,\alpha_i+\frac{1}{\be^2}\,\mu_{i,j}+o(\frac{1}{\beta^2}),
\end{eqnarray}
where $(\alpha_i)$ are the repeated eigenvalues of the matrix
$$M:=\left(\ds\int_\Gamma\frac{\partial
f_i}{\partial\nu}\overline{\frac{\partial
f_j}{\partial\nu}}\right)_{1\leq i,j\leq m}$$
in an orthonormal basis $(f_1,\cdots,f_m)$ of the eigenspace $E_\infty$.\\
Moreover, setting $P_i$ the eigenprojection associated to the eigenvalue $\alpha_i$ and $N$ the matrix given by
\begin{eqnarray}
N&:=&\big((\frac{\partial}{\partial\nu}{\cal{R}}(\frac{\partial f_i}{\partial\nu}),\frac{\partial
f_j}{\partial\nu})_{L^2(\Gamma)}+\frac{1}{1+\la_\infty}\sum_{f_k\in
E_\infty}a_{i,j,k}\nonumber\\
&+&\sum_{f_k\in
E_\infty^\perp}\frac{(1+\la_\infty)}{(1+\la_k)(\la_\infty-\la_k)}\,a_{i,j,k}\big)_{1\leq
i,j\leq m},
\end{eqnarray}
 then, $\mu_{i,j}\ , 1\leq j\leq \dim P_i $ are the
repeated eigenvalues of $P_iNP_i$ in the subspace $P_iE_\infty$.
\end{theo}

\begin{proof}
Following Kato's method (see \cite{Kato}), we introduce the
$$A_\beta:=1-P_\infty+P_\beta P_\infty=1-(P_\infty-P_\beta)P_\infty.$$
For large $\beta$ the operator $A_\beta$ is invertible and maps
$E_\infty$ onto $E_\beta$ since $\|P_\beta-P_\infty\|$ is small, and leaves the orthogonal of $E_\infty$ invariant.\\
Using Proposition \ref{proj asymp formula}, we obtain the asymptotic expansion for $A_\beta$:
\begin{eqnarray}
A_\beta=1+\frac{1}{\beta}\,QP_\infty-\frac{1}{\beta^2}\,Q_1P_\infty+o_s(\frac{1}{\beta^2})
\end{eqnarray}
Since, $P_\infty Q P_\infty=0$ it follows that,
\begin{eqnarray}
\begin{array}{lll}
A_\beta^{-1}&=&\ds 1-(\frac{1}{\beta}\,QP_\infty-\frac{1}{\beta^2}\,Q_1P_\infty)+(\frac{1}{\beta}\,QP_\infty-\frac{1}{\beta^2}\,Q_1P_\infty)^2+o_s(\frac{1}{\beta^2})\\
&=&\ds
1-\frac{1}{\beta}\,QP_\infty+\frac{1}{\beta^2}\,Q_1P_\infty+o_s(\frac{1}{\beta^2})
\end{array}
\end{eqnarray}
Now we define the operator $B_\beta$ as
$$B_\beta:=P_\infty A_\beta^{-1}(-\Delta_\beta)A_\beta P_\infty.$$
Obviously $B_\beta$ belongs to is bounded and has finite rank. Furthermore, the repeated eigenvalues of $(-\Delta_\beta)$
considered in the m-dimensional subspace $E_\beta$ are equal to
the eigenvalues of $H_\beta P_\beta$ in $E_\beta$ and
therefore also to those of $A_\beta^{-1}H_\beta A_\beta$
which is similar to $H_\beta P_\beta$ in $E_\beta$.\\
Thus, taking into account that $P_\infty Q P_\infty=0$ and that
$P_\infty$ commutes with $(-\Delta_D-z)^{-1}$, we obtain the asymptotic expansion:
\begin{eqnarray}\label{asym expan of B}
P_\infty A_\beta^{-1}(H_\beta-z)^{-1} A_\beta P_\infty
&=& (P_\infty+\frac{1}{\beta^2}P_\infty Q_1P_\infty+o_s(\frac{1}{\beta^2}))\nonumber\\
&\cdot& ((-\Delta_D-z)^{-1}+\frac{1}{\beta}LKL-\frac{1}{\be^2}(LRL-(1+z)LKLKL))\nonumber\\
& \cdot &  (1+\frac{1}{\beta}\,QP_\infty-\frac{1}{\be^2}Q_1P_\infty+o_s(\frac{1}{\beta^2}))P_\infty\nonumber\\
& = &  P_\infty
(-\Delta_D-z)^{-1}P_\infty+\frac{1}{\beta}P_\infty
LKLP_\infty\nonumber\\
&-&\frac{1}{\be^2}(P_\infty LRLP_\infty-(1+z)P_\infty
LKLKLP_\infty)\nonumber\\
&+&\frac{1}{\be^2}P_\infty Q_1(-\Delta_D-z)^{-1}P_\infty-\frac{1}{\be^2}P_\infty(-\Delta_D-z)^{-1}Q_1P_\infty \nonumber\\
&+& \frac{1}{\be^2}P_\infty LKLQP_\infty+o_u(\frac{1}{\beta^2}) \nonumber\\
& = &
(\lambda_\infty-z)^{-1}P_\infty+\frac{1}{\beta}(\lambda_\infty-z)^{-2}MP_\infty
-\ds\frac{1}{\be^2}P_\infty LRLP_\infty \nonumber\\
&+& \frac{1}{\be^2}((1+z)P_\infty LKLKLP_\infty+P_\infty
LKLQP_\infty)\nonumber\\
&+&o_u(\frac{1}{\beta^2}).
\end{eqnarray}
Here we have used the fact that $\ds
o_s(\frac{1}{\be^2})P_\infty=o_u(\frac{1}{\be^2})$ because
$P_\infty$ has finite rank.\\
Since
\begin{eqnarray}
(-\Delta_\beta)P_\beta=-\frac{1}{2i\pi}\int_{C(\lambda_\infty,\epsilon)}z\,(-\Delta_\beta-z)^{-1}\,dz,
\end{eqnarray}
integration of (\ref{asym expan of B}) along the circle
$C(\lambda_\infty,\epsilon)$ after multiplication by $(-z/2i\pi)$
and by an elementary calculation of residues at the singularity
$\la_\infty$ we obtain,
\begin{eqnarray}
B_\beta=P_\infty A_\beta^{-1}(-\Delta_\beta)P_\beta A_\beta
P_\infty=\lambda_\infty
P_\infty-\frac{1}{\beta}MP_\infty+\frac{1}{\be^2}NP_\infty+o_u(\frac{1}{\beta^2}).
\end{eqnarray}
Theorem \ref{asym expan of eigenvalues} is then a consequence of
the well known results on finite dimensional spaces (see
\cite{Kato}).

\end{proof}
\section{Example: The case of the unit disc in $\R^2$}
Let $D$ be the unit disc and ${C}$ be its boundary (the
unit circle). First, we study the solutions of the eigenvalue
problem $-\Delta f = \lambda f$ with either Dirichlet or  Neumann boundary conditions.\\
By separating variables it turns out that the solutions of the equation $-\Delta f =
\lambda f$ are given by
\begin{eqnarray}
 J_n(\sqrt{\lambda}r)e^{\pm in\theta},\ n\in\N,
\end{eqnarray}
 where the $J_n$'s are  Bessel functions of the first kind.\\
If $J_n(\sqrt{\lambda})=0$, then $\lambda$ is an eigenvalue of the
Dirichlet Laplacian on $D$ with eigenfunctions
$J_n(\sqrt{\lambda}r)e^{\pm in\theta}$. As every $J_n$ has
infinitely many positive solutions,  we shall order them as
follows $0 < k_{n,1}< k_{n,2}< \cdots <k_{n,m} < \cdots\ \ n\in\N$.\\
Therefore the eigenvalues of the Dirichlet Laplacian on the unit disc  are
given by
\begin{eqnarray}
 \lambda_{n,m}=k_{n,m}^2,\ \ n\in\N,\
m\geq 1.
\end{eqnarray}
with\ associated\ eigenfunctions
\begin{eqnarray}\label{Dir.eig.funct}
   \varphi_{n,m}^{\pm}(r,\theta)=J_n(k_{n,m}r)e^{\pm in\theta},\ \ n\in\N,\
m\geq 1.
\end{eqnarray}
The Neumann eigenvalues are characterized by the equation $\sqrt\lambda J_n'(\sqrt\lambda)=0,\ \lambda\geq 0$. As before we order the zeros of each $J_n'$ in an increasing order
\begin{eqnarray}
0 < k'_{n,1}< k'_{n,2}< \cdots <k'_{n,m} < \cdots,\ \ n\geq 1\\
0= k'_{0,1}< k'_{0,2}< \cdots <k'_{0,m} < \cdots
\end{eqnarray}
Thus the eigenvalues of the Neumann Laplacian on the unit disc are given by
\begin{eqnarray}
\mu_{n,m}=k_{n,m}'^2,\ n\in\N,\ ,m\geq 1,
\end{eqnarray}
with associated eigenfunctions,
\begin{eqnarray}
\psi_{n,m}^\pm(r,\theta)=J_n(k'_{n,m}r)e^{\pm in\theta},\ \ n\in\N,\
m\geq 1
\end{eqnarray}
By using the formula,
\begin{eqnarray}\label{bessel int}\ds\int_0^1
J_n^2(cr)r\,dr=\frac{1}{2}J_n'^2(c)+\frac{1}{2}(1-\frac{n^2}{c^2})J_n^2(c),
\end{eqnarray}
the normalized Neumann eigenfunctions associated to the eigenvalue
$\mu_{n,m}=k_{n,m}'^2$ are given by,
\begin{eqnarray}
\Psi_{n,m}^\pm(r,\theta)=\pi^{-1/2}(1-\frac{n^2}{k_{n,m}'^2})^{-1/2}\,\frac{J_n(k'_{n,m}r)}{J_n(k'_{n,m})}\,e^{\pm
in\theta},\ \ n\in\N,\ m\geq 1
\end{eqnarray}
From now on the notation $\sum_{n}\cdots(g_n^\pm,\cdot)f_n^\pm$ means
$$
\sum_{n}\cdots(g_n^+,\cdot)f_n^+ + \sum_{n}\cdots(g_n^-,\cdot)f_n^-.
$$
Thus, by the spectral calculus we obtain
\begin{eqnarray}
K_1=(-\Delta_N+1)^{-1}=\sum_{n,m}(1+k_{n,m}'^2)^{-1}(\Psi_{n,m}^\pm,\cdot)\Psi_{n,m}^\pm\\
JK_1=\sum_{n,m}\pi^{-1/2}(1+k_{n,m}'^2)^{-1}(1-\frac{n^2}{k_{n,m}'^2})^{-1/2}(\Psi_{n,m}^\pm,\cdot)e^{\pm
in\theta}\\
\label{jk1}
(JK_1)^*=\sum_{n,m}\pi^{-1/2}(1+k_{n,m}'^2)^{-1}(1-\frac{n^2}{k_{n,m}'^2})^{-1/2}(e^{\pm
in\theta},\cdot)\Psi_{n,m}^\pm,
\label{jk1Star}
\end{eqnarray}
yielding
\begin{eqnarray}
(JK_1)^*e^{\pm
in\theta}=\sum_{m\geq1}2\pi^{1/2}(1+k_{n,m}'^2)^{-1}(1-\frac{n^2}{k_{n,m}'^2})^{-1/2}\,\Psi_{n,m}^\pm\\
\|(JK_1)^*e^{\pm in\theta}\|_{L^2(D)}^2=\sum_{m\geq
1}4\pi(1+k_{n,m}'^2)^{-2}(1-\frac{n^2}{k_{n,m}'^2})^{-1}\label{norm
JK_1}
\end{eqnarray}
Let us now compute the operator $\cH$.\\
An elementary computation yields that the solution of the boundary value problem,
\begin{eqnarray}
\left\{\begin{array}{ccccc}
-\Delta u+u & = & 0 & \mbox{in} & D \\
u & = & e^{\pm in\theta} & \mbox{on} & C\\
\end{array}\right.
\end{eqnarray}
is given by,
\begin{eqnarray}
u_n(r,\theta)=\frac{J_n(ir)}{J_n(i)}\,e^{\pm in\theta}
\end{eqnarray}
Hence, the functions $e^{\pm in\theta},\ n\in\N$ belong to the domain of $\cH$, which we denote by $D(\cH)$ and
\begin{eqnarray}
\check{H}e^{\pm in\theta}=\frac{\partial
u_n}{\partial\nu}=\frac{\partial u_n(r,\theta)}{\partial
r}\rfloor_{r=1}=i\frac{J'_n(i)}{J_n(i)}\,e^{\pm in\theta}
\end{eqnarray}
That is, the eigenvalues of $\check{H}$ are
\begin{eqnarray}
\check{\lambda}_n=i\frac{J'_{n}(i)}{J_{n}(i)}\ {\rm with\ respective\ associated\  eigenfunctions}\  e^{\pm in\theta},\ \forall n\in\mathbb{N}.
\label{eigen-check}
\end{eqnarray}
Observe that each eigenvalue is a double eigenvalue except $\check{\lambda}_0$.\\
Set $L^2(C):=L^2(C,\frac{d\theta}{2\pi})$, then
\begin{eqnarray}
D(\cH)&=&\big\{\varphi\in L^2(C)\colon\, \sum_{n\in\N}\check{\lambda}_n^2|(\varphi,e^{in\theta})_{L^2(C)}|^2<\infty\big\}\nonumber\\
&& \cH\varphi = \sum_{n\in\N}\check{\lambda}_n(\varphi,e^{in\theta})_{L^2(C)}e^{in\theta},\ \forall\,\varphi\in D(\cH).
\label{domaine}
\end{eqnarray}
In other words, if we set $(c_n)_{n\in\Z}$ the Fourier coefficients of $\varphi\in L^2(C)$ and since we consider real-valued functions, then $\varphi\in D(\cH)$ if and only if
\begin{eqnarray}
\sum_{n\in\N}\check{\lambda}_n^2 |c_n|^2<\infty.
\end{eqnarray}
This observation leads to a full description of $D(\cH)$:
\begin{prop}
\begin{enumerate}
\item For each $n\in\N$, we have $n<\check{\lambda}_n<n+1/2$.
\item It follows that $\varphi\in L^2(C)$ belongs to $D(\cH)$ if and only if
$$
\sum_{n\in\N}n^2|c_n|^2<\infty.
$$
\end{enumerate}
\label{Op-check}
\end{prop}
\begin{proof}
The second assertion follows from the first one, which we proceed to prove.\\
From the recursion relations between  Bessel functions and their derivatives one has
\begin{eqnarray}
\check{\lambda}_n=i\frac{J'_{n}(i)}{J_{n}(i)}=n-i\frac{J_{n+1}(i)}{J_n(i)}\
\ \forall n\in\N.
\end{eqnarray}
Since $J_n(i)=(\frac{i}{2})^n\sum_{k=0}^\infty\frac{1}{2^{2k}k!(n+k)!}\ \
\forall n\in\N$ it follows that
\begin{eqnarray}
n<\check{\lambda}_n<n+\frac{1}{2},\ \ \forall n\in\N,
\end{eqnarray}
which finishes the proof.
\end{proof}
Now we turn our attention to compute explicitly the operators $\cH^sJK$, as they are involved in the trace-class convergence as well as in the asymptotic developments. Especially we shall prove that the limiting exponent $r=1$ in Proposition \ref{2D} is excluded.\\
Let $s\in(0,3/2]$. Relying on formulae (\ref{jk1})--(\ref{domaine}) and owing to the fact that $\cH^{3/2} JK_1$ is bounded, we obtain
\begin{eqnarray}
\check{H}^{s}JK_1&=&\ds\sum_{n\in\mathbb{N},m\geq1}\frac{2\pi^{1/2}\check{\lambda}_n^s k_{n,m}'}{(1+k_{n,m}'^2)({k_{n,m}'^2}-n^2)^{1/2}}(\Psi_{n,m}^\pm,\cdot)e^{\pm in\theta}\nonumber\\
&=&\sum_{n\in\mathbb{N},m\geq1}\check{\lambda}_n^s\theta_{n,m}(\Psi_{n,m}^\pm,\cdot)e^{\pm in\theta}
=\sum_{n\in\mathbb{N}}\check{\lambda}_n^s(\tilde{\Psi}_{n}^\pm,\cdot)e^{\pm in\theta}
\label{composition}
\end{eqnarray}
where
\begin{eqnarray}
\theta_{n,m}:=  \frac{2\pi^{1/2} k_{n,m}'}{(1+k_{n,m}'^2)({k_{n,m}'^2}-n^2)^{1/2}},\
 \tilde{\Psi}_{n}^\pm:=\sum_{m\geq1}\theta_{n,m}\Psi_{n,m}^\pm,
 %
\end{eqnarray}
Let us note that the family $\Psi_n^\pm$ is orthogonal in $L^2(D)$. Hence setting
\begin{eqnarray}
\gamma_n^2:=\|\tilde{\Psi}_{n}^\pm\|_{L^2(D)}^2=\sum_{m\geq1}\theta_{n,m}^2,\
\phi_n^{\pm}:=\gamma_n^{-1}\tilde{\Psi}_n^\pm,
\end{eqnarray}
we obtain that
\begin{eqnarray}
\check{H}^{s}JK_1=\sum_{n\in\mathbb{N}}\check{\lambda}_n^s\gamma_n(\phi_{n}^\pm,.)e^{\pm in\theta},\
(\check{H}^{s}JK_1)^* \check{H}^{s}JK_1=\sum_{n\in\mathbb{N}}\check{\lambda}_n^{2s}\gamma_n^{2}(\phi_{n}^\pm,.)\phi_n^{\pm}.
\label{AdjComp}
\end{eqnarray}
In particular we derive:
\begin{prop}
\begin{enumerate}
\item The following representation for $D_\infty$ holds true
\begin{eqnarray}
D_\infty= (-\Delta_N+1)^{-1} - (-\Delta_D + 1)^{-1}=\sum_{n\in\mathbb{N}}\check{\lambda}_n^{}\gamma_n^{}(\phi_{n}^\pm,.)\phi_n^{\pm}.
\label{DinfRep}
\end{eqnarray}
\item
$$
\lim_{\beta\to\infty}\beta\|D_\beta - D_\infty\|=\sum_{n\in\mathbb{N}}\check{\lambda}_n^{2}\gamma_n^{2}.
$$
\end{enumerate}
\end{prop}
\begin{proof}
Claim 1. is consequence of formulae (\ref{Dinf})-(\ref{AdjComp}), whereas claim 2. comes from Theorem \ref{UniConv} together with (\ref{composition}).
\end{proof}
Now we proceed to prove  that trace-class convergence with maximal rate, i.e. a rate proportional to $1/\beta$ does not hold true.
\begin{theo}
The operator $\cH JK_1$ is not a Hilbert--Schmidt operator. Consequently
\begin{eqnarray}
\lim_{\beta\to\infty}\beta\|D_\beta - D_\infty\|_{S_1}=\infty.
\end{eqnarray}
\label{negative}
\end{theo}
\begin{proof}
By \cite[Theorem 2.3-b]{bbb12}, trace-class convergence with maximal rate holds true if and only if the operator  $\check{H}JK_1$ is a  Hilbert--Schmidt operator. Hence we are led to prove tat $\|\check{H}JK_1\|_{S_2}=\infty$.\\
Let $(f_i)$ be an orthonormal basis of $L^2(D)$. As $(e^{\pm in\theta})_{n\in\N}$ is an orthonormal basis of $L^2(C)$ , we achieve
\begin{eqnarray}
\check{H}JK_1f_j&=&\sum_{n\in\mathbb{N}}(e^{\pm in\theta},\check{H}JK_1f_j)_{L^2({C})}e^{\pm in\theta}\nonumber\\
&=&\sum_{n\in\mathbb{N}}(\check{H}e^{\pm in\theta},JK_1f_j)_{L^2({C})}e^{\pm in\theta}\nonumber\\
&=&\ds\sum_{n\in\mathbb{N}}\check{\lambda}_n(e^{\pm in\theta},JK_1f_j)_{L^2({C})}e^{\pm in\theta}.
\end{eqnarray}
Yielding,
\begin{eqnarray}
\|\check{H}JK_1f_j\|_{L^2({C})}^2=\sum_{n\in\mathbb{N}}\check{\lambda}_n^{2}|
(e^{\pm in\theta},JK_1f_j)_{L^2{C})}|^2.
\end{eqnarray}
Thus
\begin{eqnarray}
\|\check{H}JK_1\|_{S_2}^2 &=&\sum_j\|\check{H}JK_1f_j\|_{L^2({C})}^2
= \sum_j\sum_n
\check{\lambda}_n^{2}|(e^{\pm in\theta},JK_1f_j)_{L^2({C})}|^2\nonumber\\
&=& \sum_n
\check{\lambda}_n^{2}\sum_j|((JK_1)^*e^{\pm in\theta},f_j)_{L^2({D})}|^2\nonumber\\
&=& \sum_n \check{\lambda}_n^{2} \|(JK_1)^*e^{\pm in\theta}\|^2_{L^2(D)}.
\end{eqnarray}
Having formula (\ref{jk1Star}) in mind we get
\begin{eqnarray}
\|\check{H}JK_1\|_2^2&=&\check{\lambda_0}^{2}
\|(JK_1)^*1\|^2_{L^2(D)}+2\sum_{n\geq1}
\check{\lambda}_n^{2} \|(JK_1)^*e^{in\theta}\|^2_{L^2(D)}\nonumber\\
&=&\check{\lambda}_0^{2}
\|(JK_1)^*1\|^2_{L^2(D)}+\sum_{n\geq1}\frac{8\pi\,\check{\lambda}_n^{2}\,k_{n,1}'^2}{(1+k_{n,1}'^2)^2 (k_{n,1}'^2-n^2)}\nonumber\\
&+&\sum_{n\geq1,\,m\geq2}\frac{8\pi\,\check{\lambda}_n^{2}\,
k'^2_{n,m}}{(1+k_{n,m}'^2)^2 (k_{n,m}'^2-n^2)}.
\end{eqnarray}
Now we have to investigate the behavior of $k_{n,m}'$ for large $n$ and $m$.\\
According to \cite{Wong}, one has for, $n,m\geq1$,
\begin{eqnarray}
n+2^{-1/3}a_m n^{1/3}<k_{n,m}<n+2^{-1/3}a_m
n^{1/3}+\frac{3}{10}a_m^2 n^{-1/3}
\end{eqnarray}
where, $a_m$ is the $m^{th}$ positive root of the equation,
\begin{eqnarray}
A_i(-x)=\frac{1}{3}\sqrt{x}(J_{1/3}(\frac{2}{3}
x^{\frac{3}{2}})+J_{-1/3}(\frac{2}{3} x^{\frac{3}{2}}))=0
\end{eqnarray}
and $A_i$ is the Airy function.\\
In the following, $c$ denotes different positive constants.\\
For large $m$ one has (see \cite{Abro}), $a_m\sim c\,m^{2/3}$. Accordingly, there exists a positive constant $c$ such that for $n,m\geq 1$,
\begin{eqnarray}
n+cm^{2/3} n^{1/3}<k_{n,m}<n+cm^{2/3} n^{1/3}+cm^{4/3} n^{-1/3}
\end{eqnarray}
On the other hand, it is known that the zeroes of $J_n$ and $J'_n$ are interlaced in the following manner:
\begin{eqnarray}
n\leq\cdots<k'_{n,m}<k_{n,m}<k'_{n,m+1}<k_{n,m+1}<\cdots
\end{eqnarray}
Hence for $n\geq1,m\geq2$ one has,
\begin{eqnarray}
n+c(m-1)^{2/3} n^{1/3}<k'_{n,m}<n+cm^{2/3} n^{1/3}+cm^{4/3}
n^{-1/3}
\label{Compnm}
\end{eqnarray}
Furthermore one has (see \cite{Abro}),
\begin{eqnarray}
k'_{0,1}&=&0, \
k'_{0,m}=k_{1,m-1}\ \forall m\geq2\ (\mbox{ since
$J'_0(z)=-J_1(z)$})\nonumber\\
&& k'_{n,1}\sim n+cn^{1/3}\ \ \mbox{for large}\ n.
\end{eqnarray}
Using the latter asymptotic together with the fact that $\check{\lambda}_n\sim n$ we obtain:
\begin{eqnarray}
\sum_{n\geq1}\frac{8\pi\,\check{\lambda}_n^{2}\,k_{n,1}'^2}{(1+k_{n,1}'^2)^2
(k_{n,1}'^2-n^2)}<\infty.
\end{eqnarray}
Consequently, $\|\check{H}^{s}JK_1\|_{S_2}$ is finite if and only if $\ds\sum_{n\geq1,\,m\geq2}\frac{n^{2}}
{k'^2_{n,m}(k_{n,m}'^2-n^2)}$ is finite.\\
However, relying on the comparison (\ref{Compnm}), we get
\begin{eqnarray*}
\begin{array}{cccc}
\ds\sum_{m\geq2}\frac{1}{k'^2_{n,m}(k_{n,m}'^2-n^2)}&\geq&\ds\sum_{m\geq2}\frac{1}
{(2n+cm^{2/3} n^{1/3}+cm^{4/3} n^{-1/3})^3(cm^{2/3}
n^{1/3}+cm^{4/3} n^{-1/3})}\\
&=&\ds\sum_{m\geq2}\frac{1} {cn^{10/3}m^{2/3}(2+cm^{2/3}
n^{-2/3}+cm^{4/3} n^{-4/3})^3(1+m^{2/3} n^{-2/3})}\\
&\geq&\ds\frac{1}{cn^{10/3}}\int_2^\infty\frac{1}
{x^{2/3}(2+cx^{2/3} n^{-2/3}+cx^{4/3} n^{-4/3})^3(1+x^{2/3}
n^{-2/3})}dx\\
&=&\ds\frac{1}{cn^3}\int_{2/n}^\infty\frac{1} {u^{2/3}(2+cu^{2/3}
+cu^{4/3})^3(1+u^{2/3})}du\\
&\sim&\ds\frac{1}{cn^3}\int_0^\infty\frac{1} {u^{2/3}(2+cu^{2/3}
+cu^{4/3})^3(1+u^{2/3})}du=\frac{c}{n^3}
\end{array}
\end{eqnarray*}
Therefore, $\ds\sum_{n\geq1,\,m\geq2}\frac{n^2}
{k'^2_{n,m}(k_{n,m}'^2-n^2)}=\infty$ and
$\|\check{H}JK_1\|_2=\infty$, which finishes the proof.
\end{proof}
%
%
%
By the end of this section we shall utilize Theorem \ref{asym expan of eigenvalues} to perform second order asymptotic for the eigenvalues of $H_\beta$. Accordingly  for $\epsilon$ small enough and  sufficiently large $\beta$, the Laplacian  with Robin boundary conditions $H_\beta$ has exactly two eigenvalues counted according to their multiplicities, for $n\geq1,\ m\geq 1$ and only one for $n=0,\ m\geq 1$ in the ball $B(k_{n,m}^2,\epsilon)$.
\begin{theo}
Set $\lambda^{(\beta)}_{n,m},\ n\in\N,\ m\geq 1$ the eigenvalues of $H_\beta$. Then
\begin{eqnarray}\label{dev asy}
\lambda^{(\beta)}_{n,m}=k_{n,m}^2-\frac{2k_{n,m}^2}{\beta}+ \frac{\alpha_{n,m}}{\be^2}         +o(\frac{1}{\beta^2}),\
n\in\N,\ m\geq 1\ {\rm for\ large}\ \beta,
\end{eqnarray}
with,
\begin{eqnarray}
\alpha_{n,m}=2ik_{n,m}^2\frac{J'_n(i)}{J_n(i)}+\frac{4k_{n,m}^4}{1+k_{n,m}^2}+%
\sum_{q\not=m}\frac{4(1+k_{n,m}^2)k_{n,m}^2k_{n,q}^2}{(1+k_{n,q}^2)(k_{n,m}^2-k_{n,q}^2)}.
\end{eqnarray}

\end{theo}
\begin{proof}
Using  formulae (\ref{Dir.eig.funct}), (\ref{bessel int}) and
the recursion relation $J'_n(z)=\frac{n}{z}J_n(z)-J_{n+1}(z)$, we obtain that the
normalized Dirichlet eigenfunctions associated to the eigenvalue
$\lambda_{n,m}=k_{n,m}^2$ are given by,
\begin{eqnarray}
f_1(r,\theta)=\pi^{-1/2}
\frac{J_n(k_{n,m}r)}{J_{n+1}(k_{n,m})}\,e^{in\theta},\ \
f_2(r,\theta)=\pi^{-1/2}
\frac{J_n(k_{n,m}r)}{J_{n+1}(k_{n,m})}\,e^{-in\theta}
\end{eqnarray}
In particular, we obtain
\begin{eqnarray}
\frac{\partial\,f_{1,2}}{\partial
r}:=\frac{\partial\,f_{1,2}(r,\theta)}{\partial
r}\rfloor_{r=1}=\pi^{-1/2}k_{n,m}\frac{J'_n(k_{n,m})}{J_{n+1}(k_{n,m})}\,e^{\pm
in\theta}=-\pi^{-1/2}k_{n,m}\,e^{\pm in\theta}.
\end{eqnarray}
Then, for $p,q\in\{1,2\}$
\begin{eqnarray}\label{f1}
(\frac{\partial\,f_p}{\partial r},\frac{\partial\,f_q}{\partial
r})_{L^2({C})}=2k_{n,m}^2\delta_{p,q}
\end{eqnarray}
Moreover,
\begin{eqnarray}
{\cal{R}}(\frac{\partial\,f_{1,2}}{\partial
r})=-\pi^{-1/2}k_{n,m}\frac{J_n(ir)}{J_n(i)}\,e^{\pm in\theta}\\
\left(\frac{\partial}{\partial r}{\cal{R}}(\frac{\partial\,f_p}{\partial
r}),\frac{\partial\,f_q}{\partial
r}\right)_{L^2(C)}=2ik_{n,m}^2\frac{J'_n(i)}{J_n(i)}\,\delta_{p,q}.
\label{f2}
\end{eqnarray}
On the other hand, if $E_{n,m}$ is the eigenspace associated to
the eigenvalue $k_{n,m}^2$, one has
\begin{eqnarray*}
E_{n,m}=\mbox{Vect}(f_1,f_2) \ \mbox{and}\
E_{n,m}^\perp=\overline{\mbox{Vect}}\left(\varphi_{p,q}(r,\theta)=\pi^{-1/2}\frac{J_p(k_{p,q}r)}{J_{p+1}(k_{p,q})}\,e^{\pm
ip\theta},\ (p,q)\not=(n,m)\right)
\end{eqnarray*}
Consequently,
\begin{eqnarray}
a_{i,j,k}=(\frac{\partial f_i}{\partial r},\frac{\partial
f_k}{\partial r})_{L^2({C})}({\frac{\partial
f_j}{\partial r},\frac{\partial f_k}{\partial
r}})_{L^2(\mathcal{C})}=4k_{n,m}^4\delta_{i,k}\delta_{j,k},
\end{eqnarray}
\begin{eqnarray}\label{f3}
\frac{1}{1+k_{n,m}^2}\sum_{f_k\in
E_{n,m}}a_{i,j,k}=\frac{4k_{n,m}^4}{1+k_{n,m}^2}\delta_{i,j},
\end{eqnarray}
\begin{eqnarray}
a_i=(\frac{\partial f_i}{\partial r},\frac{\partial
\varphi_{p,q}}{\partial
r})_{L^2(\mathcal{C})}=2k_{n,m}k_{p,q}\delta_{\pm n,\pm p},
\end{eqnarray}
and
\begin{eqnarray}\label{f4}
\sum_{\varphi_{p,q}\in
E_{n,m}}\frac{(1+k_{n,m}^2)}{(1+k_{p,q}^2)(k_{n,m}^2-k_{p,q}^2)}\,a_i
{a_j}=\sum_{q\not=m}\frac{4(1+k_{n,m}^2)k_{n,m}^2k_{n,q}^2}{(1+k_{n,q}^2)(k_{n,m}^2-k_{n,q}^2)}\,\delta_{i,j}.
\end{eqnarray}
Finally, the desired asymptotic expansion (\ref{dev asy}) is
immediate from Theorem \ref{asym expan of eigenvalues} and
formulae (\ref{f1}), (\ref{f2}), (\ref{f3}), (\ref{f4}).
\end{proof}

\bibliography{biblioLargeRobin}

\end{document}